\documentclass[11pt,reqno,twoside]{amsart}
\usepackage{latexsym}
\usepackage{a4}

\makeatletter
\def\LaTeX{\leavevmode L\raise.42ex
\hbox{\kern-.3em\size{\sf@size}{0pt}\selectfont A}\kern-.15em\TeX} \makeatother

\newtheorem{defi}{Definici\'{o}n}[section]
\newtheorem{cor}[defi]{Corollary}
\newtheorem{df}[defi]{Definition}
\newtheorem{lem}[defi]{Lemma}
\newtheorem{rem}[defi]{Remark}
\newtheorem{prop}[defi]{Proposition}
\newtheorem{thm}[defi]{Theorem}

\newcommand{\Hom}{\mbox{Hom}_{\mathcal{C}}(\tilde{T},}
\newcommand{\add}{\mbox{add}(\tau\tilde{T})}

\usepackage{color}

\begin{document}

\begin{abstract}
  We are going to show that the representation dimension of a
cluster-concealed algebra $B$ is 3. We compute its representation
dimension by showing an explicit Auslander generator for the
cluster-tilted algebra.
\end{abstract}
\vskip -.5cm

\title[Representation dimension of  cluster-concealed algebras]
{Representation dimension of cluster-concealed algebras}

\author[Gonz\'alez Chaio]{Alfredo Gonz\'alez Chaio}
\address{Depto de Matem\'atica, FCEyN, Universidad Nacional de Mar del Plata,
7600, Mar del Plata, Argentina} \email{agonzalezchaio@gmail.com}

\author[Trepode]{Sonia Trepode}
\address{Departamento de Matem\'atica, Facultad de Ciencias Exactas y Naturales,
Funes 3350, Universidad Nacional de Mar del Plata, 7600 Mar del
Plata, Argentina} \email{strepode@mdp.edu.ar}

\keywords{Auslander generator, Cluster-concealed, Cluster-tilted algebra, Representation dimension}

\maketitle

\section{Introduction}
%\vskip -.1cm In 1971,
Auslander \cite{A} introduced the concept of
representation dimension for artin algebras, motivated by the
connection of arbitrary artin algebras with representation finite
artin algebras. He expected this notion to give a reasonable way of
measuring how far an artin algebra is from being of representation
finite type. The representation dimension is a Morita-invariant of
artin algebras and characterizes the artin algebras of finite
representation type. It was shown by Auslander in \cite{A} that  an
algebra $A$ is representation-finite if and only if rep.dim$A \leq
2$. Later, Iyama proved in \cite{I} that the representation
dimension of an artin algebra is always finite,
 using a relationship with quasihereditary algebras.
The interest in representation dimension revived when Igusa and
Todorov showed that the representation dimension is related to the
finitistic dimension
 conjecture. They proved that if an artin algebra has representation
dimension at most three, then its finitistic dimension is finite
\cite{IT}. Recently, Rouquier showed in \cite{R} an exterior algebra
with representation dimension 4. In fact, he has constructed
examples of algebras with arbitrarily large representation
dimensions.
 \vskip .3cm
 On the other hand, Cluster algebras were introduced by
Fomin-Zelevinsky \cite{FZ}. Later, Marsh-Reineke-Zelevinsky
\cite{MRZ} found that there is a deep connection between cluster
algebras and quiver representations.
Buan-Marsh-Reineke-Reiten-Todorov \cite{BMRRT} defined the cluster
category and developed a tilting theory using a especial class of
objects,
 namely the cluster tilting objects. In \cite{BMR2}, Buan-Marsh-Reiten
 introduced the cluster-tilted
 algebras as the endomorphism algebras End$_\mathcal{C}(T )^{op}$
 of a cluster-tilting object
 $T$ in a cluster category $\mathcal{C}$. These algebras are
 connected to tilted algebras, which are the algebras of the form
 End$_H(T)^{op}$ for a tilting module $T$ over a hereditary
 algebra $H$. This motivates us to investigate the relationship between
 the module theory of cluster-tilted algebras and the module theory of
 hereditary algebras.
\vskip .3cm
 A cluster concealed algebra is given by $B = $
End$_\mathcal{C}(\tilde{T})^{op}$ where $T$ is a cluster tilting
object induced by a postprojective tilting $H$-module. The objective
of this paper is to compute the representation dimension of
cluster-concealed algebras by showing an explicit Auslander
generator. In order to do this, tilting and torsion theory of
hereditary algebras became very useful tools. Also the concept of
slices and local slices, the last ones defined in \cite{ABS}, became
a key tool to find Auslander generators for cluster-concealed
algebras. The notions of covariantly and contravariantly finite
categories\cite{AS} are very useful for the proof,  together with
their relationship with torsion pairs in tilted algebras \cite{AR}.
\vskip .3cm
 In section 2, we give some notations and preliminary
concepts needed for proving our main result. In section 3, we
present our main theorem and the previous results required for
proving it.

\section{preliminaries}

Through this paper we are going to use the following notation. $A$
denotes a finite dimensional algebra over an algebraically closed
field and mod $A$  represents the category of all finitely generated
right $A$-modules. $H$ denotes a finite dimensional hereditary
algebra and we denote the bounded derived category of $H$ by
$\mathcal{D}^{b}(H)$ and by [ ] the shift functor. We will often
identify the objects concentrated in degree zero with the
corresponding $H$-module.

                %%%%%%%%%% Aca vienen las subsecciones %%%%%%%%%%%%%

\subsection{Tilting theory}
\vspace{0.3cm}

We  start this section by giving the definition of tilting module.
For more details on tilting modules see  \cite{HR}. Let $A$ be  an
algebra and $T$ an $A$-module. $T$ is said to be a tilting module in
mod$A$ if it satisfies the following conditions
\begin{itemize}
\item [(a)] $pdT \leq 1$.
\item [(b)] Ext$^{1}_{A}(T, T) = 0$.
\item [(c)] $0 \rightarrow A \rightarrowT_{0} \rightarrow T_{1} \rightarrow 0$ with $T_{0}$ and $T_{1}$ in add$T$.
\end{itemize}

If a module satisfies condition (b), we say that the module is
exceptional, or equivalently, we say that $T$ is a tilting module if
$pdT \leq 1$, is rigid and has $n$ non-isomorphic indecomposable
direct summands, where $n$ is the number of non-isomorphic simple
modules.

\begin{rem}
We recall that for a hereditary algebra $H$, an $H$-module $T$ is
said to be a tilting module if $T$ is rigid and has $n$
non-isomorphic indecomposable direct summands.
\end{rem}

A tilting module is said to be basic if all of its direct summands
are non-isomorphic. The endomorphism ring of a tilting module over a
hereditary algebra is said to be a $\mathbf{tilted}$ $\mathbf{
algebra}$. In particular, hereditary algebras are tilted algebras.
\vskip .1cm
%%%%  Slices  %%%%
   We recall that a $\mathbf{path}$
from $X$ to $Y$ is a sequence of indecommposable modules and
non-zero morphisms $X = X_0 \rightarrow X_1 \rightarrow ...
\rightarrow X_t = Y$.  Given $X,Y \in \mbox{ind}A$, we say that $X$
is a $\mathbf{predecessor}$ of $Y$ or that $Y$ is a
$\mathbf{successor}$ of $X$, provided that there exists a path from
$X$ to $Y$. A tilting module $T$ is $\mathbf{convex}$ if, for a
given pair of indecomposable summands of $T$, $X$ $Y$  in add $T$,
any path from $X$ to $Y$ contains only indecomposable modules in add
$T$. Following \cite{APT}, we say that a set $\Sigma_{T}$ in mod $A$
is a $\mathbf{complete}$ $\mathbf{slice}$ if
 $T = \bigoplus_{M \in \Sigma_{T}}M$ is a convex tilting module
with End$_A T$ hereditary. For the original definition of complete
slices, we refer to \cite{Ri1}, \cite{Ri2}. Tilted algebras are
characterized by the existence of a complete slice in its module
category.

\vskip .1cm
 For a given tilting module in mod$H$ there exist two
full disjoint subcategories of mod$H$, namely

$$\mathcal{F}(T)=\{X \in \mbox{mod}H \mbox{ such that } \mbox{Hom}_H(T,X)=0 \}$$
$$\mathcal{T}(T)=\{X \in \mbox{mod}H \mbox{ such that } \mbox{Ext}^{1}_H(T,X)=0 \}$$
\vskip .1cm The free torsion class and torsion class respectively.
\vskip .1cm
 Furthermore, if $T$ is a convex tilting, then
mod$H=\mathcal{F}(T)\bigcup \mathcal{T}(T)$. We have that, in this
case,  $\mathcal{F}(T)$ is closed under predecessors and
$\mathcal{T}(T)$ is closed under successors.

\subsection{Cluster categories and cluster tilted Algebras}

For the convenience of the reader, we start this section recalling
some definitions and results of cluster categories from
\cite{BMRRT}. Let $\mathcal{C}$ be the  $\mathbf{cluster}$
$\mathbf{category}$ associated to $H$ and given by
$\mathcal{D}^{b}(H)/F$, where $F$ is the composition functor
$\tau_{\mathcal{D}}^{-1}[1]$. We  represent by $\tilde{X}$ the class
of an object $X$ of $\mathcal{D}^{b}(H)$ in the cluster category. We
recall that Hom$_{\mathcal{C}}(\tilde{X},\tilde{Y})= \bigoplus_{i
\in \mathbb{Z}}\mbox{Hom}_{\mathcal{D}^b(H)}(X,F^iY).$ We recall
that $S = \mbox{ind}H \bigcup H[1]$ is the fundamental domain of
$\mathcal{C}$. If $X$ and $Y$ are objects in the fundamental domain,
then we have that $\mbox{Hom}_{\mathcal{D}^b(H)}(X,F^iY)=0$ for all
$i \neq 0,1$. Moreover, any object in $\mathcal{C}$ is of the form
$\tilde{X}$ with $X \in S$. \vskip .1cm For $\mathcal{C}$, we say
that $\tilde{T}$ in $\mathcal{C}$ is a $\mathbf{tilting}$
$\mathbf{object}$ if Ext$^{1} _{\mathcal{C}}(\tilde{T}, \tilde{T}) =
0$ and $\tilde{T}$ has a maximal number of non-isomorphic direct
summands. A tilting object in $\mathcal{C}$ has finite summands.
\vskip .1cm

There exists the following nice correspondence between tilting
modules and basic tilting objects. \vskip .1cm
 \noindent$
\mathbf{Theorem}$ \cite[Theorem 3.3.]{BMRRT} {\it\begin{itemize}
\item[(a)] Let T be a basic tilting object in $\mathcal{C} = \mathcal{D}^b(H)/F$, where $H$
is a hereditary algebra with $n$ simple modules.
\begin{itemize}
\item[(i)] $T$ is induced by a basic tilting module over a hereditary algebra $H'$,
derived equivalent to $H$.
\item[(ii)] $T$ has $n$ indecomposable direct summands.
\end{itemize}
\item[(b)] Any basic tilting module over a hereditary algebra $H$ induces a basic tilting
object for $\mathcal{C} = \mathcal{D}^b(H)/F$.
\end{itemize}}
\vskip .1cm
 In \cite{BMR2}, Buan, Marsh and Reiten introduced the
cluster-tilted algebra. Let $\tilde{T}$ be a tilting object over the
cluster category $\mathcal{C}$, we recall that $B$ is the
$\mathbf{cluster}$ $\mathbf{tilted}$ $\mathbf{algebra}$ if $B=$
End$_{\mathcal{C}}(\tilde{T})$. It is also shown in \cite{BMR2}
that, if $\tilde{T}$ is a tilting object in $\mathcal{C}$, the
functor $\Hom \,\,)$ induces an equivalence of categories between
$\mathcal{C}/\add$ and mod$B$.

Thus using the equivalence above, we can compute the Hom$_B(X',Y')$
in terms of the cluster category $\mathcal{C}$.
$$Hom_{B}(\Hom \tilde{X}),\Hom \tilde{Y}))\simeq \mbox{Hom}_{\mathcal{C}}(\tilde{X},\tilde{Y})/\add,$$ where $X'= \Hom \tilde{X})$ and $Y'= \Hom \tilde{Y})$ are $B$-modules.

\begin{rem} \label{2.4}
Note that, if $X' = \Hom \tilde{X})$ and $Y'= \Hom \tilde{Y})$ ,
then we have that $Hom_B(X',Y')=0$ if for every $f:\tilde{X}
\rightarrow \tilde{Y}$, $f$ factors through $\add$ in $\mathcal{C}$.
\end{rem}

\subsection{Covariantly and contravariantly categories}

In this section, we recall some facts on approximation morphisms and
covariantly and contrava-riantly finite categories from \cite{AS},
also see \cite{AR}.

First we begin by recalling the notion of approximations. Let
$\mathcal{X}$ be a full subcategory of mod$A$ closed under direct
summands and isomorphisms. We say that an A-module $Y$ is
contravariantly finite over $\mathcal{X}$ if there exists a morphism
$f:X \rightarrow Y$, with $X \in \mathcal{X}$, such that
Hom$_A(X',f):\mbox{Hom}_A(X',X) \rightarrow \mbox{Hom}_A(X',Y)$ is
surjective for all $X' \in \mathcal{X}$. We say that $Y$ is
covariantly finite over $\mathcal{X}$ if there exists a morphism
$g:Y \rightarrow X$, with $X \in \mathcal{X},$ such that
Hom$_A(g,X'):\mbox{Hom}_A(X,X') \rightarrow \mbox{Hom}_A(Y,X')$ is
surjective for all $X' \in \mathcal{X}$. Furthermore, we  say that
the morphism $f:X \rightarrow Y$ is a right
$\mathcal{X}$-approximation of $Y$ and that the morphism $g:Y
\rightarrow X$ is a left $\mathcal{X}$-approximation of $Y$. A
morphism $f:B \rightarrow C$ is right minimal if $g:B \rightarrow B$
an endomorphism, and $f=fg$ then $g$ is an automorphism. A right
$\mathcal{X}$-approximation $f:X \rightarrow Y$ of $Y$ is minimal if
$f$ is right minimal. Two right minimal $\mathcal{X}$-approximation
are isomorphic.

Moreover, by \cite[Proposition 3.9]{AS}, an $A$-module $Y$ is
contravariantly finite over $\mathcal{X}$ if and only if there
exists a right $\mathcal{X}$-approximation of $Y$ and an $A$-module
$Y$ is covariantly finite over $\mathcal{X}$ if and only if there
exists a left $\mathcal{X}$-approximation of $Y$.

%% covariantly finite categories%%%% PREGUNTAR

The category $\mathcal{X}$ is called covariantly (contravariantly) finite in mod$A$
if every $A$-module $Y$ has a left (right) minimal $\mathcal{X}$-approximation.

Recall from \cite{AR} that Gen$(T)=\mathcal{T}(T)$ is covariantly finite and
$\mathcal{F}(T)=$Cogen$(\tau T)$ is contravariantly finite.

\subsection{Representation dimension}
We recall that an $A$-module $M$ is a generator for mod$A$ if for
each $X \in$ mod$A$, there exists an epimorphism $M' \rightarrow X $
with $M' \in $add$(M)$. Observe that $A$ is a generator for mod$A$.
Dually, we say that an $A$-module $M$ is a cogenerator if for each
$Y \in $ mod$A$ there exists a monomorphism $Y \rightarrow M'$ with
$M' \in $add$(M)$. Note that $DA$ is a cogenerator for mod$A$. In
particular, any module $M$ containing every indecomposable
projective and every indecomposable injective module as a summand is
a generator-cogenerator module for mod$A$.

The original definition of representation dimension (we will note it
by rep.dim) of an artin algebra $A$ is due to Auslander. For more
facts on this topic, we refer the reader to \cite{A}. The following
is a nice characterization of representation dimension, in the case
that $A$ is a non semisimple algebra, also due to Auslander. This
characterization is given as follows
$$ \mbox{rep.dim}A = \mbox{inf} \,\{ \mbox{gl.dim End}_A(M) / M \mbox{is a generator-cogenerator for mod} A \}$$

A module $M$ that reaches the minimum is called an
$\mathbf{Auslander \, generator}$ and $\mbox{gl.dim End}_A(M)=
\mbox{rep.dim}A$ if $M$ is an Auslander generator.

The representation dimension can also be defined in a functorial
way, which will result us more convenient. The next definition (see
\cite{APT},\cite{CP},\cite{R}) will be very useful for the rest of
this work.

\begin{df}\label{df. repdim}
The representation dimension rep.dim$A$ is the smallest integer $i
\geq 2$ such that there is a module $M \in \mbox{mod}A$ with the
property that, given any $A$-module $X$ \begin{itemize}
\item [(a)], there is an exact sequence of
$$0 \rightarrow M^{-i+2} \rightarrow M^{-i+3} \rightarrow ... \rightarrow M^{0}  \stackrel{f}\rightarrow X \rightarrow 0$$

with $M^{j} \in \mbox{add}(M)$ such that the sequence

$$ 0 \rightarrow \mbox{Hom}_{A}(M, M^{-i+2}) \rightarrow ... \rightarrow  \mbox{Hom}_{A}(M, M^{0}) \rightarrow \mbox{Hom}_{A}(M, X) \rightarrow 0$$

is exact.
\item [(b)] there is a exact sequence
$$0 \rightarrow X \stackrel{g}\rightarrow M'_{0} \rightarrow M'_{1} \rightarrow ... \rightarrow M'_{i-2} \rightarrow 0$$

with $M'_{j} \in \mbox{add}(M)$ such that the sequence

$$ 0 \rightarrow \mbox{Hom}_{A}(M'_{-i+2},M) \rightarrow ... \rightarrow  \mbox{Hom}_{A}(M'_{0}, M) \rightarrow \mbox{Hom}_{A}(X, M) \rightarrow 0 $$

is exact.
\end{itemize}
\end{df}
following \cite{CP}, we say the module $M$ has the $i$-resolution property and that the sequence $0 \rightarrow M^{-i+2} \rightarrow M^{-i+3} \rightarrow ... \rightarrow M^{0}  \rightarrow L \rightarrow 0$ is an add$M$-approximation of $L$ of length $i$. Note, that $f: M_0 \rightarrow X$ is a right add$(M)$-approximation of $X$ and $g:X \rightarrow M'_0$ is a left add$(M)$-approximation of $X$.

\begin{rem}\label{1}
Note that either condition (a) or (b) implies that $\mbox{gl.dim End}_A(M)$  $\leq i+1$. Then, if $M \in \mbox{mod}A$ and $i \geq 2$, the following statements are equivalent
\begin{itemize}
\item $M$ satisfies (a) and (b) of the definition.
\item $M$ satisfies (a) and $M$ contains an injective cogenerator as a direct summand.
\item $M$ satisfies (b) and $M$ contains a projective generator as a direct summand.
\end{itemize}
\end{rem}

 We have rep.dim$A=$rep.dim$A^{op}$.% Para que esta esto?

%%%%%%%%%%%%%%%%%%%%%%%%%%%%%%%%%%%%%%%%%%%%%%%%%%%%%%%%%%%%%%%%%%%%%%%%%%%%%%%%%%%%%%%%%%%%%%%%%%%%%%%%%%%%%%%%%%%%%%%%%%%%%%%%%%%%%%%%%%%%%%%%%%%%%%%%%%%%%%%%%%%%%%%%%%

%PARECE que ya esta!!!!!!%%%%%%%%%%%%%%%%%%%%%%%%%%%%%%%%%%%%%%%%%%%%%%%%%%%%%%%%%%%

\section{Representation dimension for cluster concealed algebras}

\vspace{0.3cm}

In this section we  present our main result. Let $\tilde{T}$ be a
tilting object in a cluster category $\mathcal{C}$ and let be
$B=$End$_{\mathcal{C}}(\tilde{T})$, the associated cluster-tilted
algebra. To simplify some proofs we  choose without lose of
generality
 $T$ and $\tau T$ without projective summands.  As we are
interested in compute the representation dimension of cluster
concealed algebras, through this section $T$ will be a
postprojective tilting module on mod$H$ and $\tilde{T}$  denotes the
class of $T$ on the cluster category $\mathcal{C}$. Recall that
$\tilde{T}$ is a tilting object in $\mathcal{C}$. We consider $H$,
and thus $B$, of infinite representation type.

We consider a complete slice $\Sigma$ on the postprojective
component of the hereditary algebra $H$, such that $ \Sigma \in
\mathcal{T}(T)$. Note that  this condition implies $\Sigma$ does not
have projective summands. We can construct a convex tilting module
of the form $U = \bigoplus_{E \in \Sigma}E$, and we  denote
$\mathcal{T}(\Sigma)$ to the category $\mathcal{T}(U)$ and
$\mathcal{F}(\Sigma)$ to $\mathcal{F}(U)$. Since $\Sigma \in
\mathcal{T}(T)$ also $U \in \mathcal{T}(T)$ and thus we have
$\mathcal{T}(\Sigma) \subset \mathcal{T}(T)$. Then we have $\Sigma'
= \Hom \tilde{\Sigma})$, the image of $\tilde{\Sigma}$ by the
BMR-equivalence \cite{BMR2}, which is a local slice in mod$B$
\cite{ABS}.

Using the BMR-equivalence, we can describe the indecomposable modules in mod$B$ as follows.
Let $Y$ be an indecomposable in mod$B$,  then we have one of the following cases
\begin{itemize}
\item[(1)] $Y \simeq \Hom Q)$ with $Q = \widetilde{P[1]}$ with $P$ projective indecomposable in mod$H$.
\item[(2)] $Y \simeq \Hom \tilde{G})$ with $G \in \mathcal{F}(\Sigma) \subset \mbox{mod}H$.
\item[(3)] $Y \simeq \Hom \tilde{X})$ with $X \in \mathcal{T}(\Sigma) \subset \mbox{mod}H$.

\end{itemize}

\vspace{0.3cm}

Observe that if $Y$ is indecomposable also $Q,X$ or $G$ is
indecomposable. Since we choose $\Sigma$ in the postprojective
component of mod$H$, we have that $\mathcal{F}(\Sigma)$ has only a
finite number of non isomorphic indecomposable $H$-modules, then we
only have a finite number of isomorphism classes of indecomposable
$B$-modules $Y$ satisfying the second case. As well as in the first
case, because we only have a finite number of indecomposable
projective $H$-modules. This implies we can consider the following
modules, $Q'= \Hom \widetilde{H[1]})$ and
$\mathcal{G}=\bigoplus_{G}\Hom G)$ with $G$ indecomposable in
$\mathcal{F}(\Sigma)$, in mod$B$.

We want to determine an Auslander generator for mod$B$, this is a module $M$ with the
properties of Definition \ref{df. repdim}, so we want a module which provides
add$(M)$-approximations for every one of the modules in the previous cases. Since by
the Remark \ref{1}, a module $M$ satisfies Definition \ref{df. repdim} if satisfies
condition $(a)$ and contains a cogenerator as a direct summand, we will focuses on
find modules that satisfies condition $(a)$.
If we let $M'$ to be the module $\mathcal{G} \oplus Q'$, clearly, $M'$ approximates
trivially the modules of the first and second case.

Hence we will concentrate from now on, in the modules of the third case.
For a given complete slice $\Gamma$ we consider the set ${\mathcal{D}_{\Gamma}}$ defined as
$${\mathcal{D}_{\Gamma}}= \{ Y \in \mbox{mod}B \mbox{ such that } Y = \Hom \tilde{X})
 \mbox{ with } X \in \mathcal{T}(\Gamma) \},$$

that is the modules in ${\mathcal{D}_{\Sigma}}$ are the ones of the third case.
Our objective now is to approximate the modules of ${\mathcal{D}_{\Sigma}}$ by $\Sigma'$.

Let $Y \in {\mathcal{D}_{\Sigma}}$, then we have
$Y = \mbox{Hom}_{\mathcal{C}}(\tilde{T},\tilde{X})$ with $X \in \mathcal{T}(\Sigma) \subset \mbox{mod}H$.

Now, $X \in \mathcal{T}(\Sigma)=\mbox{Gen}(\Sigma)$ then there exist an epimorphism $f$
$$f:E \rightarrow X \rightarrow 0$$
\vspace{0.1cm}

\noindent a right $\mbox{add}(\Sigma)$-approximation of $X$.
Recall that a complete slice on a tilted algebra induces a convex tilting module, then since $H$ is hereditary and $X \in \mbox{Gen}(\Sigma)$ by \cite[Proposition :)]{APT}, we have that $K = \mbox{Ker}f \in \mbox{add}(\Sigma)$ and hence that the exact short sequence

$$0\rightarrow K \rightarrow E \rightarrow X \rightarrow 0  $$
\vspace{0.1cm}

is an $\mbox{add}(\Sigma)$-resolution of length two for $X$. We want to construct a similar exact sequence in mod$B$ for $Y$. The idea is to prove that the image of this sequence in $\mathcal{C}$ induces such a sequence in mod$B$.

This short exact sequence induces a triangle in $\mathcal{D}^{b}(H)$

$$ K \rightarrow E \rightarrow X \rightarrow K[1]$$
\vspace{0.1cm}

which induces a triangle in the cluster category $\mathcal{C}$ by taking the respective quotient classes

$$ \tilde{K} \rightarrow \tilde{E} \rightarrow \tilde{X} \rightarrow \tau\tilde{K} \, \, \, (*_1)$$
\vspace{0.1cm}

Our objective is to study the image of this triangle in mod $B$ by the functor $\Hom \, \, )$

\begin{lem}\label{3.1}
Let $E$ be a module in add$(\Sigma)$ and $X \in \mathcal{T}(\Sigma)$. Then Hom$_\mathcal{C}(\tilde{X},\tau \tilde{E})= \mbox{Hom}_{\mathcal{D}^{b}(H)}(X,E[1])$.
\end{lem}

\begin{proof}
Recall that $\Sigma$ do not have projective summands, hence, since $E$ is in add$(\Sigma)$, $E$ is not projective and thus $\tau E$ is in mod$H$.

$$\mbox{Hom}_{\mathcal{C}}(\tilde{X},\tau\tilde{E})=\bigoplus_{i \in Z}\mbox{Hom}_{\mathcal{D}^{b}(H)}(X,F^{i}\tau E).$$

since $X$ and $\tau E$ are $H$-modules, thus $\tilde{X}$ and $\tau\tilde{E}$ are in the fundamental domain of $\mathcal{C}$ therefore

$$\mbox{Hom}_{\mathcal{C}}(\tilde{X},\tau\tilde{E})=
\mbox{Hom}_{\mathcal{D}}(X,\tau E) \oplus \mbox{Hom}_{\mathcal{D}}(X,E[1])$$
\vspace{0.2cm}

\noindent where all the others summands are zero.

We are going to see that also
$\mbox{Hom}_{\mathcal{D}^{b}(H)}(X,\tau E)$ is zero.

 Now $ \mbox{Hom}_{\mathcal{D}^{b}(H)}(X,\tau E)=\mbox{Hom}_{H}(X,\tau E)$ , since $X$ and $\tau E$ are $H$ modules.
Moreover, since $E \in$ add$(\Sigma)$ we have that $\tau E \in \mathcal{F}(\Sigma)$ and $X \in \mathcal{T}(\Sigma)$.

Hence $\mbox{Hom}_{\mathcal{D}^{b}(H)}(X,\tau E) = \mbox{Hom}_{H}(X,\tau E)= 0$ because  $(\mathcal{F},\mathcal{T})$ is a  torsion pair.
\end{proof}

we are in conditions to show that the modules in ${\mathcal{D}_{\Sigma}}$ are generated by modules in $\mbox{add}(\Sigma')$.

\begin{prop} \label{prop}
Let $T$ be a postprojective tilting module over mod$H$. Let $\Sigma$
be a complete slice on the postprojective component of mod$H$ such
that $\Sigma$ is contained on  $\mathcal{T}(T)$. Then
\begin{itemize}
  \item[(i)] ${\mathcal{D}_{\Sigma}} \subset \mbox{Gen}({\Sigma'})$ in mod$B$.
  \item[(ii)] Moreover, consider $\tau^{-1}\Sigma$, wich is also a complete slice, there exist a short exact sequence of the form
  $$ 0\rightarrow K' \rightarrow E' \rightarrow Y \rightarrow 0$$
  with $E',K' \in \mbox{add}(\Sigma')$ for every $Y \in {\mathcal{D}_{\tau^{-1}\Sigma}}$.
\end{itemize}
\end{prop}

\begin{proof}

Let $Y \in {\mathcal{D}_{\Sigma}}$, recall we have a triangle $ \tilde{K} \rightarrow \tilde{E} \rightarrow \tilde{X} \rightarrow \tau\tilde{K}$ in $\mathcal{C}$

and recall the functor  $\mbox{Hom}_{\mathcal{C}}(\tilde{T},\, \, \,)$ induces an equivalence between $\mathcal{C}/\add $ and $\mbox{mod}B$.

Now, applying this functor to the last triangle, we obtain a long exact sequence in mod$B$

$$ ... \rightarrow \mbox{Hom}_{\mathcal{C}}(\tilde{T}, \widetilde{X[-1]}) \rightarrow \mbox{Hom}_{\mathcal{C}}(\tilde{T},\tilde{K})\rightarrow \mbox{Hom}_{\mathcal{C}}(\tilde{T},\tilde{E})\rightarrow $$

$$\rightarrow\mbox{Hom}_{\mathcal{C}}(\tilde{T},\tilde{X})\rightarrow \mbox{Hom}_{\mathcal{C}}(\tilde{T},\tau\tilde{K}) \rightarrow ...   \, \, \, (*)      $$
\vspace{0.3cm}

We want to construct an epimorphism from add$(\Sigma')$ to $Y \in \mathcal{D}_{\Sigma}$, we know that $\Sigma' = \Hom \tilde{\Sigma})$ and $E$ is an add$(\Sigma)$-approximation.
 Consider $E'= \mbox{Hom}_{\mathcal{C}}(\tilde{T},\tilde{E})$ and $f^{*} = \mbox{Hom}_{\mathcal{C}}(\tilde{T},\tilde{f})$. We are going to prove that $f^{*}:E'\rightarrow Y$ is an epimorphism. In order to do this, it suffices to see that

\noindent$\mbox{Hom}_{B}(\mbox{Hom}_{\mathcal{C}}(\tilde{T},\tilde{X}),\mbox{Hom}_{\mathcal{C}}(\tilde{T},\tau \tilde{K}))= 0$.

Since $\mbox{Hom}_{B}(\mbox{Hom}_{\mathcal{C}}(\tilde{T},\tilde{X}), \mbox{Hom}_{\mathcal{C}}(\tilde{T},\tau\tilde{K}))\simeq \mbox{Hom}_{\mathcal{C}}(\tilde{X},\tau\tilde{K})/ \mbox{add}(\tau\tilde{T})$, then consider $h:\tilde{X} \rightarrow \tau\tilde{K}$ in $\mathcal{C}$, then  by remark \ref{2.4}, if $h$ factorize by $\mbox{add}(\tau\tilde{T})$, we have that $\mbox{Hom}_{B}(\mbox{Hom}_{\mathcal{C}}(\tilde{T},\tilde{X}),\mbox{Hom}_{\mathcal{C}}(\tilde{T},\tau\tilde{K}))= 0$.

%Without lose of generality, assume $K$ is an indecomposable module. If not we can consider an indecomposable summand and the results holds for %additivity.

By lemma \ref{3.1}, since $K \in$ add$(\Sigma)$ we can assume that
$h \in \mbox{Hom}_{\mathcal{D}^{b}(H)}(X,K[1])$. Since we need to
prove that $h$ factorizes by add$(\tau \tilde{T})$ we are going to
consider $F\tau T = T[1]$, then is enough to show that $h$
factorizes by $T[1]$ in $\mathcal{D}^{b}(H)$.

We can assume without lose of generality that there exist $H'$ hereditary, derived equivalent to $H$ such that $X, T[1]$ and $K[1]$ are $H'$ modules, in fact, we can choose $H'=$End$_H(\tau^{-1}\Sigma)$.

Then $T[1]$ results a tilting module over $H'$ and we may consider the category $\mathcal{T}(T[1])$ which is a full subcategory of mod$H'$.

Under the hypothesis taken, we know that $K \in $ add$(\Sigma) \subset \mathcal{T}(\Sigma) \subset \mathcal{T}(T)$. Then we have $K[1] \in \mathcal{T}(T[1])$ and since $X$ is in mod$H$ we also have
$0=\mbox{Hom}_{\mathcal{D}^{b}(H)}(T[1],X) \simeq \mbox{Hom}_{\mathcal{D}^{b}(H')}(T[1],X)=\mbox{Hom}_{H'}(T[1],X)$ because $T[1]$ and $X$ are $H'$-modules. Thus we have shown that $X \in \mathcal{F}(T[1])$.

Therefore if $h: X \rightarrow K[1]$ we observe that, since $\mathcal{T}(T[1])$ is covariantly finite \cite{AR} and $\mathcal{T}(T[1])=$GEN$(T[1])$, there exist a left minimal $\mathcal{T}(T[1])$-approximation of $X$  $g: X \rightarrow T'$ with $T' \in \mbox{add}(T[1])$ such that Hom$_{H'}(g,K[1]): \mbox{Hom}_{H'}(T',K[1]) \rightarrow \mbox{Hom}_{H'}(X,K[1])$ is surjective. Hence there exist $k:T' \rightarrow K[1]$ such that $kg=h$.

 We have the following commutative diagram:

$$\begin{array}{ccc}
  X & \stackrel{h}\longrightarrow & K[1] \\
   & \searrow \, & \uparrow \\
   &  &   T'
\end{array}$$

 it is,  $h$ should factorizes by add$(T[1])$, thus $h:\tilde{X} \rightarrow \tau\tilde{K}$ factorizes by $\mbox{add}(\tau \tilde{T})$.

Therefore we have that the long exact sequence in $(*)$ factors through zero
$$ \begin{array}{cccccccc}
  &  \mbox{Hom}_{\mathcal{C}}(\tilde{T},\tilde{E}) & \stackrel{f^*}\longrightarrow & \mbox{Hom}_{\mathcal{C}}(\tilde{T},\tilde{X}) & \longrightarrow & \mbox{Hom}_{\mathcal{C}}(\tilde{T},\tau\tilde{K}) & \longrightarrow & ... \\
      &  &  &   & \searrow \,\,\, \nearrow &    &  &  \\
      &  &  &  & 0 &  &  &
\end{array}$$

 and so $f^{*}:\mbox{Hom}_{\mathcal{C}}(\tilde{T},\widetilde{E}) \rightarrow \mbox{Hom}_{\mathcal{C}}(\tilde{T},\widetilde{X})$ is an epimorphisms with $E'=\mbox{Hom}_{\mathcal{C}}(\tilde{T},\widetilde{E}) \in \mbox{add}(\Sigma')$.

This finishes the proof of (i).

Only remains to see that $\mbox{Hom}_{\mathcal{C}}(\tilde{T},\widetilde{X[-1]}) \rightarrow \mbox{Hom}_{\mathcal{C}}(\tilde{T},\tilde{K})$ is zero also. Then, as we did before, we want to see that

 $\mbox{Hom}_{B}(\mbox{Hom}_{\mathcal{C}}(\tilde{T},\widetilde{X[-1]}), \mbox{Hom}_{\mathcal{C}}(\tilde{T},\tilde{K})) = 0$, then we are going to prove that any $h \in \mbox{Hom}_{\mathcal{C}}(\widetilde{X[-1]},\tilde{K})$ factors through add$(\tau\tilde{T})$. Observe that $\widetilde{X[-1]}=\tau^{-1}\tilde{X}$ then we have that

\begin{center}
 $\mbox{Hom}_{\mathcal{C}}(\tau^{-1}\tilde{X},\tilde{K})= \mbox{Hom}_{\mathcal{D}^{b}(H)}(\tau^{-1}X, K) \oplus
\mbox{Hom}_{\mathcal{D}^{b}(H)}(\tau^{-1}X,\tau^{-1}K[1])$
 \end{center}

 where $\mbox{Hom}_{\mathcal{D}^{b}(H)}(\tau^{-1}X, K) \simeq \mbox{Hom}_{\mathcal{D}^{b}(H)}(X,\tau K)= \mbox{Hom}_{H}(X,\tau K) =0 $ because $X \in \mathcal{T}(\Sigma)$ and $\tau K \in \mathcal{F}(\Sigma)$.

 Therefore $\mbox{Hom}_{\mathcal{C}}(\tau^{-1}\tilde{X},\tilde{K}) = \mbox{Hom}_{\mathcal{D}^{b}(H)}(\tau^{-1}X,\tau^{-1}K[1]) \simeq \mbox{Hom}_{\mathcal{D}^{b}(H)}(X, K[1])$, then we know by (i) that any $h$ in $\mbox{Hom}_{\mathcal{C}}(\tau^{-1}\tilde{X},\tilde{K})$ will factor through add$(\tau\tilde{T})$ in $\mathcal{C}$ and we have

 $$0 \rightarrow \mbox{Hom}_{\mathcal{C}}(\tilde{T},\tilde{K})\rightarrow \mbox{Hom}_{\mathcal{C}}(\tilde{T},\tilde{E})\rightarrow \mbox{Hom}_{\mathcal{C}}(\tilde{T},\tilde{X})\rightarrow 0$$

\end{proof}

 \vspace{0.3cm}

Therefore, if we let $M$ be the $B$-module $M' \oplus \Sigma'$ our last proposition shows that we have $M$ generates all modules in mod$B$ because add$(\Sigma')$ is contained in add$(M)$. More over, our next lemma will prove that the epimorphism cons-tructed before is an add$(\Sigma')$-approximation of the modules in ${\mathcal{D}_{\Sigma}}$.

\begin{lem}\label{lem 1}
Let $E'=\mbox{Hom}_{\mathcal{C}}(\tilde{T},\widetilde{E})$.
Then, if $Y \in {\mathcal{D}_{\Sigma}}$  the epimorphism $f^{*}:E'
\rightarrow Y$, constructed before, is an
$\mbox{add}(\Sigma')$-approximation of $Y$.
\end{lem}

\begin{proof}
Let $U'$ be an indecomposable module in  $\mbox{add}(\Sigma')$. We want to show that, $\mbox{Hom}_B(U',E') \rightarrow  \mbox{Hom}_B(U',Y) \rightarrow 0$ is exact, that is, if $0 \neq h \in \mbox{Hom}_{B}(U',Y)$ there exists $k \in \mbox{Hom}_{B}(U',Y)$ such that $f^{*}k=h$.

Using the [BMR]-equivalence we can assume there exists $X \in \mathcal{T}(T)$, $U \in \mbox{add}(\Sigma)$ such that $U'= \Hom \tilde{U})$ and $Y = \Hom \tilde{X})$. Hence we can write, $h \in \mbox{Hom}_{B}(U',Y) = \mbox{Hom}_{B}(\mbox{Hom}_{C}(\tilde{T},\tilde{U}), \mbox{Hom}_{C}(\tilde{T},\tilde{X})) \simeq \mbox{Hom}_{C}(\tilde{U},\tilde{X})/\mbox{add}(\tau\tilde{T})$, then there exist $h' \in \mbox{Hom}_{C}(\tilde{U},\tilde{X})$ such that $h=\mbox{Hom}_{C}(\tilde{T},h')$. Since $h \neq 0$ we can assume that $h'$ don't factorize by the $\mbox{add}(\tau\tilde{T})$ in $\mathcal{C}$.
Let's compute $\mbox{Hom}_{\mathcal{C}}(\tilde{U}, \tilde{X})$

$$\mbox{Hom}_{\mathcal{C}}(\tilde{U}, \tilde{X})= \mbox{Hom}_{\mathcal{D}^{b}(H)}(U, X) \oplus \mbox{Hom}_{\mathcal{D}^{b}(H)}(U, FX)$$
\vspace{0.3cm}

Since $U$ and $X$ are taken in the fundamental domain, even more $U$
is postprojective($U$ is not contained in any oriented cycle), we
know that at most one of the last terms is different from zero,
\cite{BMR2}.

Observe that $\mbox{Hom}_{\mathcal{D}^{b}(H)}(U, FX)= \mbox{Hom}_{\mathcal{D}^{b}(H)}(U, \tau^{-1}X[1])= \mbox{Ext}^{1}_H(U, \tau^{-1}X)= \mathcal{D}\mbox{Hom}_H(\tau^{-1}X,\tau U)$, and $\mbox{Hom}_H(\tau^{-1}X,\tau U)=0$ because $\tau^{-1}X \in \mathcal{T}(\Sigma)$ and $U \in \mathcal{F}(\Sigma)$.

Then we have only one case, $h' \in \mbox{Hom}_{\mathcal{D}^{b}(H)}(U, X)$.

Suppose $h' \in \mbox{Hom}_{\mathcal{D}^{b}(H)}(U, X)$ then $h' \in \mbox{Hom}_{H}(U, X)$ because $U$ and $X$ are $H$-modules,  and then, since $f:E \rightarrow X $ is an $\mbox{add}(\Sigma)$-approximation of $X \in \mbox{mod}(H)$, there exist $k': U \rightarrow E$ such that the following diagram commutes

$$\begin{array}{ccc}
   &  & U \\
   & \swarrow & \downarrow \\
  E & \rightarrow &  X
\end{array}$$

it is, $fk'=h'$. This commutative diagram induces a commutative diagram on $\mathcal{D}^{b}(H)$, which induces the following commutative diagram in $\mathcal{C}$

$$\begin{array}{ccc}
   &  & \tilde{U} \\
   & \swarrow & \downarrow \\
  \tilde{E} & \rightarrow &  \tilde{X}
\end{array}$$

Again, applying $\mbox{Hom}_{\mathcal{C}}(\tilde{T}, \, )$ we obtain

$$\begin{array}{ccc}
   &  & \mbox{Hom}_{\mathcal{C}}(\tilde{T},\tilde{U}) \\
   & \swarrow & \downarrow \\
  \mbox{Hom}_{\mathcal{C}}(\tilde{T},\tilde{E}) & \rightarrow &  \mbox{Hom}_{\mathcal{C}}(\tilde{T},\tilde{X})
\end{array}$$

Taking $k = \mbox{Hom}_{\mathcal{C}}(\tilde{T},k')$ we have that $f^{*}k = \mbox{Hom}_{\mathcal{C}}(\tilde{T},\tilde{f}k') = \mbox{Hom}_{\mathcal{C}}(\tilde{T},h')=h$.
Then, $h: U' \rightarrow Y$ factors through $E'$; observe that $k \neq 0$, otherwise $h'$ would factorize by $\mbox{add}(\tau\tilde{T})$.
\vspace{0.3cm}
\end{proof}

We are going to see now that this approximation is in fact, an
add$(M)$-approximation of the modules in $\mathcal{D}_\Sigma$; To
achieve this objective we are going to prove that the morphism
$\mbox{Hom}_B(M',E')
\stackrel{\mbox{Hom}_B(M',f^{*})}\longrightarrow
\mbox{Hom}_B(M',E')$ is an epimorphism. Recall $M' = \mathcal{G}
\bigoplus Q'$. Before we establish our next result, we are going to
state the next technical lemma.

%el (equal to) de la proxima demostracion en un lema aparte, por que no es tan claro y molesta en la proxima demostración%

\begin{lem}\label{cg}
For any indecomposable $G$ in $\mathcal{F}(\Sigma)$ and any indecomposable $X$ in $\mathcal{T}(\Sigma)$
$$ \mbox{Hom}_{\mathcal{C}}(\tilde{G},\tilde{X})= \mbox{Hom}_{H}(G,X).$$
\end{lem}

\begin{proof}

We know that $G$ and $X$ are indecomposable $H$-modules. We can assume that $\tilde{G}$ and $\tilde{X}$ are in the fundamental domain of $\mathcal{C}$.
Therefore we can compute $\mbox{Hom}_{\mathcal{C}}(\tilde{G},\tilde{X})= \mbox{Hom}_{\mathcal{D}^{b}(H)}(G,X)\oplus \mbox{Hom}_{\mathcal{D}^{b}(H)}(G,F(X))$, because since $X$ and $G$ are in the fundamental domain,  we know by \cite{BMRRT}, that only this summands can be different from zero.

We want to see that $\mbox{Hom}_{\mathcal{D}^{b}(H)}(G,F(X)) = 0$.

If $X$ is injective then $\tau^{-1}X = P[1]$ with $P$ a projective indecomposable module in $\mbox{mod}H$, so $F(X)= \tau^{-1}X[1]= P[2]$ and $ \mbox{Hom}_{\mathcal{D}^{b}(H)}(G,F(X))= \mbox{Hom}_{\mathcal{D}^{b}(H)}(G,P[2]) = \mbox{Ext}^{2}_{H}(G,P)= 0$, because $H$ is hereditary.

If $X$ is not injective, then we have that $\tau^{-1}X \in \mbox{mod}H$ and then $$ \mbox{Hom}_{\mathcal{D}^{b}(H)}(G,\tau^{-1}X[1])= \mbox{Ext}^{1}_{H}(G,\tau^{-1}X) \simeq D\mbox{Hom}_{H}(\tau^{-1}X,\tau G)$$ where the last isomorphism is given by the Auslander-Reiten formula, but we have, if $G$ is not projective, $\tau G \in \mathcal{F}(\Sigma)$ and $\tau^{-1}X \in \mathcal{T}(\Sigma)$ because $\mathcal{T}(\Sigma)$ is closed under successors and $\mathcal{F}(\Sigma)$ is closed under predecessors, so $\mbox{Hom}_{H}(\tau^{-1}X,\tau G) = 0 $ and therefore $\mbox{Hom}_{\mathcal{D}^{b}(H)}(G,\tau^{-1}X[1]) = 0$.

if $G$ is projective we have $\mbox{Ext}^{1}_{H}(G,\tau^{-1}X) \simeq D\mbox{Hom}_{H}(\tau^{-2}X,G)=0$, assuming $\tau^{-2}X$ not injective, otherwise $\tau^{-2}X= P[1]$ and $\mbox{Ext}^{2}_{H}(G,P)=0$.

Then we have proof that $\mbox{Hom}_{\mathcal{D}^{b}(H)}(G,F(X)) = 0$, and so
$\mbox{Hom}_{\mathcal{C}}(\tilde{G},\tilde{X}) = \mbox{Hom}_{\mathcal{D}^{b}(H)}(G,X) = \mbox{Hom}_{H}(G,X)$ because $G$ and $X$ are $H$-m\'odules.
\end{proof}

We are now in conditions, of proving that $f^{*}:E' \rightarrow Y$
is an add$(\Sigma \oplus \mathcal{G})$-approximation of $Y$.

\begin{lem}\label{lem 2}
Let $f^{*}$ be the morphism constructed before. If $C_G$ is in add$(\mathcal{G})$, then $\mbox{Hom}_B(C_G, f^{*}):\mbox{Hom}_B(C_G, E') \rightarrow \mbox{Hom}_B(C_G, Y)$ is an epimorphism.
\end{lem}
\begin{proof}

Let $C_{G}=\mbox{Hom}_{C}(\tilde{T},\tilde{G})$ with $G \in \mathcal{F}(\Sigma)$ indecomposable.
Suppose that there is a map $0 \neq h : C_{G}\rightarrow Y$, we want to show that the following diagram:

$$\begin{array}{ccc}
   &  & C_G \\
   & \swarrow & \downarrow \\
  E' & \rightarrow &  Y
\end{array}$$

commutes, it is, there exists a $ k: C_{G}\rightarrow E'$ in $\mbox{mod}(B)$ such that $fk=h$.
We can write, as we did before, using the [BMR]-equivalence $h \in \mbox{Hom}_{B}(C_G,Y) \simeq \mbox{Hom}_{C}(\tilde{G},\tilde{X})/\mbox{add}(\tau\tilde{T})$, then there exist $h' \in \mbox{Hom}_{C}(\tilde{G},\tilde{X})$ such that $h=\mbox{Hom}_{C}(\tilde{T},h')$. We can assume that $h'$ don't factorize by the $\mbox{add}(\tau\tilde{T})$ in $\mathcal{C}$, otherwise $h$ will be zero in mod$B$.

So, we have that, $\mbox{Hom}_{B}(\mbox{Hom}_{C}(\tilde{T},\tilde{G}),\mbox{Hom}_{C}(\tilde{T},\tilde{X}))\simeq \mbox{Hom}_{C}(\tilde{G},\tilde{X})/\mbox{add}(\tau\tilde{T})$. By lemma \ref{cg} since $G \in \mathcal{F}(\Sigma)$ and $X \in \mathcal{T}(X)$ $ \mbox{Hom}_{C}(\tilde{G},\tilde{X})= \mbox{Hom}_{H}(G,X)$ and $h':G \rightarrow X$.

Hence, we have that $h=\mbox{Hom}_{\mathcal{C}}(\tilde{T},h')$ with $h' \in \mbox{Hom}_{H}(G,X)$. Then, since $f:E \rightarrow X $ is an $\mbox{add}(\Sigma)$-approximation of $X \in \mbox{mod}H$ and $\Sigma$ is a complete slice in mod$H$, any morphism from $G$ to $X$, will factorize by $E$ because $\mathcal{T}(\Sigma)$ is covariantly finite. This is, there exist $k': G \rightarrow E$ such that the following diagram commutes

$$\begin{array}{ccc}
   &  & G \\
   & \swarrow & \downarrow \\
  E & \rightarrow &  X
\end{array}$$

We may now proceed as in lemma \ref{lem 1}, and show that this diagram induces a commutative diagram on mod$B$

$$\begin{array}{ccc}
   &  & \mbox{Hom}_{\mathcal{C}}(\tilde{T},\tilde{G}) \\
   & \swarrow & \downarrow \\
  \mbox{Hom}_{\mathcal{C}}(\tilde{T},\tilde{E}) & \rightarrow &  \mbox{Hom}_{\mathcal{C}}(\tilde{T},\tilde{X})
\end{array}$$

 We have  proved that there exist $k= \mbox{Hom}_{\mathcal{C}}(\tilde{T},\tilde{k'})$ such that $f^{*}k = \mbox{Hom}_{\mathcal{C}}(\tilde{T},\tilde{f}\tilde{k'}) = \mbox{Hom}_{\mathcal{C}}(\tilde{T},\tilde{h'})=h$.
 Then, $h: C_{G} \rightarrow Y$ factors through $E'$; observe that $k \neq 0$ because $k'$ don't factorize by $\mbox{add}(\tau\tilde{T})$, otherwise $h'$ would factorize by $\mbox{add}(\tau\tilde{T})$.
 This finishes proof of the lemma.
\end{proof}

It only rest to prove that, if we have a map from a module $Q$ in
add$(Q')$, in any module $Y \in \mathcal{D}_\Sigma$, this map also factors
through $E'$. Our next lemma proves the desired result.

\begin{lem}\label{lem 3}
Let $h$ be a map in $\mbox{Hom}_{B}(Q,Y)$, then there exist a
morphism $k$ such that $f^{*}k=h$ and $h$ factors through $E'$.
\end{lem}
\begin{proof}

Assume we have $0 \neq h \in \mbox{Hom}_{B}(Q,Y)$ where
$Q=\mbox{Hom}_{\mathcal{C}}(\tilde{T}, \widetilde{P[1]})$ with $P$
an indecomposable projective module in $H$.

Let $h$ be, $h=
\mbox{Hom}_{\mathcal{C}}(\tilde{T}, h')$ with $h' \in
\mbox{Hom}_{\mathcal{C}}(\widetilde{P[1]},\tilde{X})$. Then, we have

$$ \mbox{Hom}_{\mathcal{C}}(\widetilde{P[1]},\tilde{X})= \oplus_{i \in Z}\mbox{Hom}_{\mathcal{D}^{b}(H)}(P[1],F^{i}X)=$$
$$\mbox{Hom}_{\mathcal{D}^{b}(H)}(P[1],X) \oplus \mbox{Hom}_{\mathcal{D}^{b}(H)}(P[1],\tau^{-1}X[1])$$

where,
$\mbox{Hom}_{\mathcal{D}^{b}(H)}(P[1],X)=0$ because $X \in$ mod$H$
and $P[1] \in$ mod$H[1]$ therefore there is no non zero map between
$X$ and $P[1]$ in $D^{b}(H)$.

Then the only possibility for $h' \neq 0$ is $h' \in \mbox{Hom}_{\mathcal{D}^{b}(H)}(P[1],\tau^{-1}X[1])$ and so $h'[-1] \in \mbox{Hom}_{\mathcal{D}^{b}(H)}(P,\tau^{-1}X)$.

First observe that if $X$ is injective then $\tau^{-1}X =
P'[1]$ with $P'$ projective indecomposable in mod$H$, therefore
$\mbox{Hom}_{\mathcal{D}^{b}(H)}(P,\tau^{-1}X)=
\mbox{Hom}_{\mathcal{D}^{b}(H)}(P,P'[1])=
\mbox{Ext}^{1}_{H}(P,P')=0$.

Let us assume then, that $X$ is not injective, hence $\tau^{-1}X \in$ mod$H$. Moreover, we have $h'[-1] \in \mbox{Hom}_{H}(P,\tau^{-1}X)$.

Since $X \in \mathcal{T}(\Sigma)$ and $\mathcal{T}(\Sigma)$ is closed under successors we have, $\tau^{-1}X \in \mathcal{T}(\Sigma)$. Moreover, we can consider $\tau^{-1} \Sigma$ as an complete slice, and so $\tau^{-1}X \in \mathcal{T}(\tau^{-1}\Sigma)$.

Recall $f:E \rightarrow X$ was constructed as an $\mbox{add}(\Sigma)$- approximation of $X$, for $X \in \mathcal{T}(\Sigma)$. Now we can take
$$\tau^{-1}f:\tau^{-1}E \rightarrow \tau^{-1}X \rightarrow 0.$$ We
now can construct the following commutative diagram on mod$H$

$$\begin{array}{ccccc}
   &  & P \\
    &  &  \downarrow  \\
\tau^{-1}E & \stackrel{\tau^{-1}f}\longrightarrow & \tau^{-1}X &
\rightarrow & 0
\end{array}$$

Since $P$ is projective, then the morphism $h'[-1]$ factors through
$\tau^{-1}E$, it is, there exists $k':P \rightarrow \tau^{-1}E$ such
that the following diagram

$$\begin{array}{ccc}
 P  & \longrightarrow & \tau^{-1}X \\
  \, \searrow &  & \nearrow \\
 & \tau^{-1}E &
\end{array}$$
commutes.

Applying the shift functor $[\,\,\,]$ to this diagram, we obtain a new
commutative diagram in $\mathcal{D}^{b}(H)$.

$$\begin{array}{ccccc}
 P[1]  & \longrightarrow  &\tau^{-1}X[1]=FX \\
\searrow & & \nearrow\\
 & \tau^{-1}E[1] = FE &
\end{array}$$

This diagram, induces the following diagram on $\mathcal{C}$

$$\begin{array}{ccc}
 \widetilde{P[1]}  & \longrightarrow & \tilde{X} \\
  \searrow &  & \nearrow \ \\
 & \tilde{E} &
\end{array}$$

And, as we did before we apply the functor $\mbox{Hom}_{\mathcal{C}}(\tilde{T}, \,\,\, )$ to obtain a commutative diagram on mod$B$

$$\begin{array}{ccc}
   &  & \mbox{Hom}_{\mathcal{C}}(\tilde{T}, \widetilde{P[1]})\\
   & \swarrow & \downarrow \\
  \mbox{Hom}_{\mathcal{C}}(\tilde{T},\tilde{E}) & \rightarrow &  \mbox{Hom}_{\mathcal{C}}(\tilde{T},\tilde{X})
\end{array}$$

Then, if we choose $k \simeq
\mbox{Hom}_{\mathcal{C}}(\tilde{T},\widetilde{k'[1]})$, we get that
$f^{*}k=h$, where the commutativity is given by the above diagram.

\end{proof}

\begin{prop} Let $M$ be the $B$-module, $M = \Sigma' \oplus Q' \oplus \mathcal{G} $ where $\mathcal{G}$ , $Q'$ and $\Sigma$ are considered as before. Then the map $Hom_B(N,f^{*}):\mbox{Hom}_{B}(N,E')\rightarrow \mbox{Hom}_{B}(N,Y)$ is an epimorphism for every module $Y \in D_{\Sigma}$ and $N \in \mbox{add}(M)$  and $E'$ the $\mbox{add}(\Sigma')$-approximation constructed before.
\end{prop}

\begin{proof}
Let $N$ be an indecomposable direct summand of $M$.
Then we can consider the following cases:

\begin{itemize}
\item[(a)]
$ N \in \mbox{add}(\Sigma')$
\item[(b)]
$ N \in \mbox{add}(Q')$
\item[(c)]
$ N \in \mbox{add}(\mathcal{G})$
\end{itemize}

Hence we know the results holds for $N$ in all the above cases. Case $(a)$ follows from Lemma \ref{lem 1}, case $(b)$ from Lemma \ref{lem 2} and case $(c)$ from Lemma \ref{lem 3} Then the proposition follows for $M$ by the additivity of the functor
$\mbox{Hom}_{B}(N, \, \,)$.

\end{proof}

This proves that there exist an add$(M)$-resolution of length at
most two, for every module in mod$B$. Note that this proposition
proves that $M$ satisfies the condition (a) of definition \ref{df.
repdim} for $i=2$ and therefore that gl.dim End$_B(M) \leq 3$. We
are going to proceed now to state and prove our main result.

\begin{thm}
Let $B$ a cluster-concealed algebra of infinite representation type. Then $\mbox{rep.dim}B = 3$.
\end{thm}

\begin{proof}
Let $M$ be the module in the above proposition. Observe that if $I$
is injective in mod$B$, then $I =
\mbox{Hom}_{\mathcal{C}}(\tilde{T},\tau^{2} \tilde{T})$, but
$\tau^{2}T \in \mathcal{F}(\Sigma)$, recall $\tau T$ does not have
projective summands and $\mathcal{F}(T) \subset \mathcal{F}(\Sigma)$
due to the hypothesis taken on $\Sigma$ on $H$ , therefore any
indecomposable injective module in mod$B$ is a direct summand of $M$
The same occurs for any indecomposable $P \in \mbox{mod}B$, then $P
= \mbox{Hom}_{B}(\tilde{T},\tilde{T})$ and again $\tilde{T}$ lies in
$\mathcal{F}(\Sigma)$, thus the $\mbox{add}(\mathcal{G})$ contains every
indecomposable projective and every indecomposable injective module.
Then $\mathcal{G}$ is generator-cogenerator for mod$B$ and hence $M$ is a
generator-cogenerator for mod$B$. Moreover, proposition proves that
gl.dim End$_B(M) \leq 3$ and thus rep.dim$B \leq 3$.
\end{proof}

%%%%Aplications%%%%%%%%%%

We are going to see some applications of our main theorem. Let $A$
be basic connected finite dimensional algebra, $A$ is said to be of
minimal infinite type if it is of infinite representation type and
for each vertex $e$ in the quiver of $A$ we have that $A/AeA$ is of
finite type. Recall that for any cluster-tilted algebra $B$ with a
vertex $e$ in the quiver of $B$ we have that $B/BeB$ is again a
cluster-tilted algebra(see \cite{BMR3} section 2).

So, as a consequence of the theorem, we have the following
corollary.

\begin{cor}
Let $B$ minimal cluster-tilted of infinite representation type. Then $\mbox{rep.dim}B = 3$.
\end{cor}

\end{document}